\documentclass[12pt,notitlepage]{amsart}
\usepackage{latexsym,amsfonts,amssymb,amsmath,amsthm}
\usepackage{graphicx}
\usepackage{color}
\usepackage[inner=1.0in,outer=1.0in,bottom=1.0in, top=1.0in]{geometry}

\newcommand{\mz}{\ensuremath{\mathbb Z}}
\newcommand{\mr}{\ensuremath{\mathbb R}}
\newcommand{\mh}{\ensuremath{\mathbb H}}
\newcommand{\mq}{\ensuremath{\mathbb Q}}
\newcommand{\mc}{\ensuremath{\mathbb C}}

\newcommand{\half}{\ensuremath{ \frac{1}{2}}}

\theoremstyle{plain}		
	\newtheorem{mytheo}{Theorem} [section]
	
	\newtheorem{myprop}[mytheo]{Proposition}
	
     \newtheorem{mylemma}[mytheo]{Lemma}

\theoremstyle{remark}

\numberwithin{equation}{section}
\numberwithin{figure}{section}

\begin{document}
\title{Zeros of the weight two Eisenstein Series}
\author{Rachael Wood}
\address{Harding University \\ Searcy \\ AR  72149 \\ U.S.A.}
\email{rwood5@harding.edu}
\author{Matthew P. Young} 
\address{Department of Mathematics \\
	  Texas A\&M University \\
	  College Station \\
	  TX 77843-3368 \\
		U.S.A.}
\email{myoung@math.tamu.edu}

\thanks{This work was conducted in summer 2013 during an REU conducted at Texas A\&M University.  The authors thank the Department of Mathematics at Texas A\&M and the NSF for supporting the REU.
In addition, this material is based upon work of M.Y. supported by the National Science Foundation under agreement No. DMS-1101261.  Any opinions, findings and conclusions or recommendations expressed in this material are those of the authors and do not necessarily reflect the views of the National Science Foundation.}

\begin{abstract}
We develop some of the finer details of the location of the zeros of the weight two Eisenstein series.  These zeros are the same as the zeros of the derivative of the Ramanujan delta function.
\end{abstract}

\maketitle

\section{Introduction}
The zero sets of modular forms have attracted much attention over the years.  In an influential paper, 
F. Rankin and Swinnerton-Dyer \cite{RS} showed that the zeros of the weight $k$ Eisenstein series $E_{k}(z)$ that lie in the standard fundamental domain lie on the circle $|z|=1$.  In sharp contrast, Rudnick \cite{Rudnick} showed the zeros of Hecke cusp forms become equidistributed in the fundamental domain as the weight becomes large, conditionally on the mass equidistribution conjecture which has since been proven by Holowinsky and Soundararajan \cite{HolowinskySound}.  
There have been a number of other results on zeros of various types of modular forms; see for example \cite{Hahn} \cite{MNS} \cite{DukeJenkins}  \cite{GhoshSarnak}.  In all of these cases,
one observes a dichotomy with the zeros either being patterned (say, with all zeros lying on some lower-dimensional set), or equidistributed in some larger space (which one could view as randomness).

One naturally wonders about the location of zeros of the derivative of a modular form (which is not a modular form, but instead is a so-called quasimodular form).  In contrast to the modular case, the zero set of a quasimodular form is not well-defined on $\Gamma \backslash \mathbb{H}$ (where $\Gamma$ is the relevant congruence subgroup).  This complicates matters, because there may be infinite sequences of zeros approaching the boundary of the upper half plane--there is no reduction theory that maps the zeros back to a fundamental domain for $\Gamma \backslash \mh$.

In this paper, we provide a fairly comprehensive description of the zero set of the weight $2$ Eisenstein series $E_2(z)$ defined by the Fourier expansion
\begin{equation}
\label{eq:E2def}
E_2(z)=1-24\displaystyle\sum\limits_{n=1}^\infty\sigma_1(n)e(nz).
\end{equation}
It is known that $E_2(z) = \frac{1}{2 \pi i} \frac{\Delta'}{\Delta}(z)$ (following from the product formula for $\Delta$), so the zeros of $E_2$ in $\mh$ agree with the zeros of $\Delta'$.
El Basraoui and Sebbar \cite{BS} have shown that the weight $2$ Eisenstein series has infinitely
many $SL_2(\mz)$-inequivalent zeros within the strip $G = \{z \in \mh: -\tfrac12 \leq x \leq \tfrac12 \}$.
This was generalized to other quasimodular forms by 
Saber and Sebbar \cite{SaSe}; also see
Balasubramanian and Gun \cite{BG} for some related results.  El Basraoui and Sebbar \cite{BS} have also identified certain infinite families of translates of the standard fundamental domain that all contain (or do not contain, respectively) a zero of $E_2$.

\section{Background}
Let $D$ denote the closure of the standard fundamental domain, i.e., 
\begin{equation}
D=\{z\in \mathbb{H} : |z| \geq 1   \text{ and }  -\tfrac{1}{2} \leq  x \leq \tfrac{1}{2}   \}.
\end{equation}
The weight $k>2$ (even) Eisenstein series for $\Gamma = PSL_2(\mz)$ is defined by
\begin{equation}
\label{eq:Eksum}
 E_k(z) = \sum_{\gamma \in \Gamma_{\infty} \backslash \Gamma} (j(\gamma, z))^{-k},
\end{equation}
where $\Gamma_{\infty}$ is the stabilizer of $\infty$, and $j(\gamma, z) = cz+d$ if $\gamma = (\begin{smallmatrix} a & b \\ c & d \end{smallmatrix})$.  It has the Fourier expansion
 \begin{equation}
 \label{eq:EkFourier}
 E_{k}(z)= 1 +\gamma_{k}  \sum_{n=1}^{\infty}
 \sigma_{k-1}(n)e(nz),
 \end{equation}
 where
 $ \gamma_{k}= -  \frac{2k}{B_k}$, 
 $B_k$  is the $k${th} Bernoulli number, and
  $\sigma_{k-1}(n) = \sum_{a |n} a^{k-1}$.  For $k=2$, the sum in \eqref{eq:Eksum} does not converge absolutely, consistent with the fact that there are no non-zero modular forms of weight $2$ for $SL_2(\mz)$.
When $k=2$, the function $E_2(z)$ defined by \eqref{eq:EkFourier}
is a quasimodular form, which satisfies the relation
\begin{equation}
\label{eq:E2relation}
 E_2\Big(\frac{az+b}{cz+d}\Big)=(cz+d)^2 E_2(z)-\tfrac{6}{\pi}ic(cz+d).
\end{equation}
See Zagier's chapter \cite{Zagier} for a proof of the transformation property of $E_2$, as well as a general definition of a quasimodular form.  Note that $E_2$ is periodic with period $1$, and for this reason we typically restrict attention to the strip $-\half \leq x \leq \half$.

\section{Numerical experiments on the zeros of $E_2(z)$}
\label{section:experiments}
We used Mathematica to numerically solve the equation $E_2(z)=0$ for $y\geq \varepsilon$ for various values of $\varepsilon$, using a truncated Fourier expansion to approximate $E_2(z)$.  This requires more than $ \varepsilon^{-1}$ terms in the Fourier expansion, so this becomes unwieldly for very small $\varepsilon > 0$.  For reasons that will be clear momentarily, it is natural to order the zeros by decreasing $y$-values.  Figure \ref{fig:Listplotofzeros} shows some zeros computed by Mathematica.

\begin{figure}[h]
\includegraphics[width=0.7\textwidth]{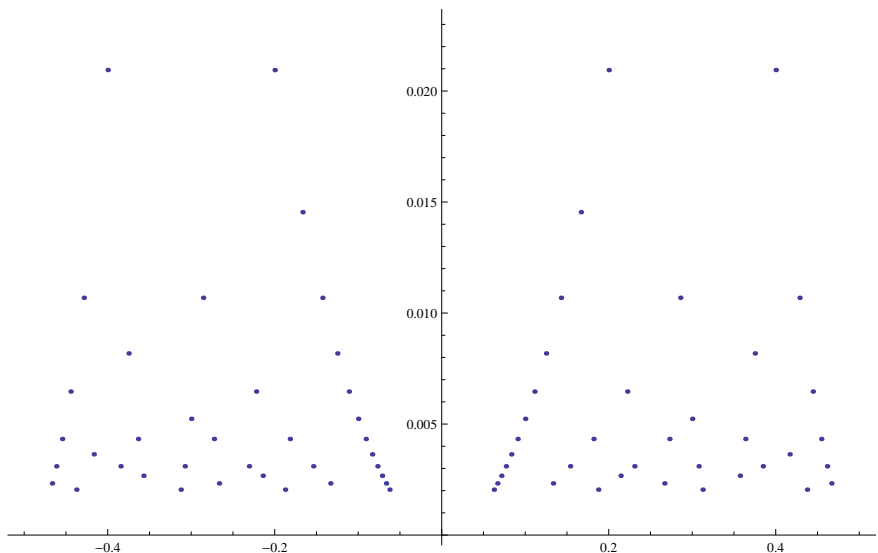}
\caption{Zeros of $E_2$ for $.002 < y <.022$}
\label{fig:Listplotofzeros}
\end{figure}

The highest zero say $z_1$ occurs at $x=0$, $y = 0.5235217000179992\dots$.  It is easy to show that there exists a unique zero on the line $x=0$ because $E_2(iy)$ is real, increasing for $y > 0$, and has limits $\lim_{y \rightarrow 0^+} E_2(iy) = -\infty$, $\lim_{y \rightarrow \infty} E_2(iy) = 1$.  

There is also a zero say $z_2$ at $x= - 1/2$, $y= 0.13091903039676245\dots$.  We now briefly prove that there indeed exists a zero with $x = - 1/2$, following \cite{BS}.  It is clear from the Fourier expansion that $E_2(- 1/2 + iy)$ is real.  Next let $\gamma = (\begin{smallmatrix} 1 & 0 \\ 2 & 1 \end{smallmatrix})$, so for $z = -\half + iy$, 
\begin{equation}
 E_2(-\tfrac12 + \tfrac{i}{4y}) = E_2\big(\frac{-\half + iy}{2iy}\big) = E_2(\gamma z) = -4 y^2 E_2(-\tfrac12 + iy) + \tfrac{24}{\pi} y.
\end{equation}
Taking $y \rightarrow \infty$, the right hand side is $\sim -4y^2 + \frac{24}{\pi} y$, which has the limit $-\infty$.  On the other hand, $\lim_{y \rightarrow 0^+} E_2(-\half + \tfrac{i}{4y}) = 1$.  Thus $E_2(-1/2 + iy)$ has a zero for $y \in \mr$ by the intermediate value theorem.  This type of method to show existence of zeros on certain vertical lines only works on $x=0$ and $x = \pm 1/2$ since $E_2(x+iy)$ is real-valued only on these lines.  Since the Fourier coefficients of $E_2$ are real-valued, if $E_2(x+iy) = 0$, then $E_2(-x+iy) = 0$, so we shall sometimes only display the zeros with $-\half < x < 0$.
%
%

The next few approximate zeros with $x < 0$ produced by Mathematica are:
\begin{equation}
\label{eq:zerotable}
\begin{aligned}
z_3 &= -0.33332589074451363 +  0.058181923654001474 i\\
z_4 &= -0.2499951743678368 + 0.03272491502475048 i \\
z_5^1 &= -0.19999706592873248 + 0.020942992286928155 i \\
z_5^2 &= -0.40000182048192795 +  0.020946451276672513 i.
\end{aligned}
\end{equation}
One immediately notices the striking fact that the $x$-coordinates of the zeros are quite close (but not equal) to the rational numbers $-1/3$, $-1/4$, $-1/5$, $-2/5$.  The $y$-coordinates also have a pattern:
\begin{gather*}
 \frac{\text{Im}(z_1)}{\text{Im}(z_2)} = 3.99882\dots, \quad \frac{\text{Im}(z_1)}{\text{Im}(z_3)} = 8.99801\dots, \quad 
 \frac{\text{Im}(z_1)}{\text{Im}(z_4)} = 15.9976 \dots, \\
  \frac{\text{Im}(z_1)}{\text{Im}(z_5^1)} = 24.9975 \dots, \quad 
  \frac{\text{Im}(z_1)}{\text{Im}(z_5^2)} = 24.9933 \dots,
\end{gather*}
which are all close (but not equal) to squares of integers.

Needless to say, the data suggest that these patterns continue indefinitely.
One of the main goals of this paper is to explain these patterns.  We have
\begin{mytheo}
\label{thm:zeroapprox}
 Let $-d/c \in [-1/2,1/2]$ be a rational number with $(c,d) = 1$, 
 and 
 set
 \begin{equation}
 \widehat{z}_c^{d} = -\frac{d}{c} + \frac{i}{c^2 (6/\pi)}.  
 \end{equation}
For each such rational number, there exists a zero $z_{c}^{d}$ of $E_2$ satisfying
 \begin{equation}
  |z_c^d - \widehat{z}_c^{d}| \leq \frac{.000283}{c^2 (6/\pi)^2}.
 \end{equation}
\end{mytheo}
For ease of comparison, we have
$\widehat{z}_1^0 = 0 + .523599 i$, $\widehat{z}_2^1 = -\tfrac12 + .130899 i$, $\widehat{z}_3^1 = -\tfrac13 + .0581776 i$, 
 $\widehat{z}_4^1 = -\tfrac14 + .0327249 i$, $\widehat{z}_5^1 = -\tfrac15 + .020944 i$,  $\widehat{z}_5^2 = -\tfrac25 + .020944 i$, etc., which explains the numerical patterns noticed above.  Furthermore, the points $\widehat{z}_c^d$ with $d$ fixed and varying $c$ lie on the parabola $y = \frac{\pi}{6 d^2} x^2$, which explains the apparent families of curves appearing in Figure \ref{fig:Listplotofzeros}.

In Theorem \ref{thm:locationofzeros} below we show in a more precise sense that the zeros of $E_2$ are naturally in one-to-one correspondence with the rational numbers.
 

After writing this paper, we learned that Imamo{\=g}lu, Jermann, and T{\'o}th \cite{IJT} simultaneously and independently have obtained results substantially the same as our Theorem \ref{thm:zeroapprox}.

\section{An equivariant function}
Our main tool for understanding the zeros of $E_2$ is an auxiliary function $h(z)$ defined by
\begin{equation}
 h(z)=z+\frac{6/\pi i}{E_2(z)}.
\end{equation}
This function $h(z)$ is \emph{equivariant}, which means that
if $z \in \mathbb{H}$ and $\gamma \in SL_2(\mathbb{Z})$, then
\begin{equation}
\label{eq:hequivariant}
h(\gamma z)=\gamma h(z).
\end{equation}
See Sebbar-Sebbar \cite{SeSe} for some properties of $h$, especially their Theorem 5.3 which proves \eqref{eq:hequivariant}.  Note that $h(z_0) = \infty$ is equivalent to $E_2(z_0) = 0$.  El Basraoui and Sebbar \cite{BS} originally used properties of $\frac{1}{h(z)}$ to show that $E_2(z)$ has infinitely many $\Gamma$-inequivalent zeros in the strip $G$.  We include a brief proof of \eqref{eq:hequivariant} since it is not a well-known property.
\begin{proof}[Proof of \eqref{eq:hequivariant}]
 Let
 $\gamma = (\begin{smallmatrix} a & b \\ c & d \end{smallmatrix}) \in SL_2(\mathbb{Z})$ and $z \in \mh$.
 By \eqref{eq:E2relation}, we have
 \begin{equation}
  h(\gamma z) = \frac{1}{cz + d} \Big[ az + b + \frac{\frac{6}{\pi i}}{(cz + d) E_2(z) + \frac{6}{\pi i} c} \Big].
 \end{equation}
Dividing both the numerator and denominator of the second term by $E_2(z)$ and forming a common denominator, we obtain
\begin{equation}
 h(\gamma z) = \frac{1}{cz+d}\left[\frac{(az+b) (cz+d+\frac{6c}{\pi i E_2(z)})+\frac{6}{\pi i E_2(z)}}{(cz+d) + \frac{6c}{\pi i E_2(z)}}\right].
\end{equation}
Using $bc + 1= ad$ in the numerator, and factoring out $cz+d$ from the numerator, we have
\begin{equation}
 h(\gamma z) = \frac{az + b +  a\frac{6}{\pi i E_2(z)}}{cz + d + c\frac{6}{\pi i E_2(z)}} = \gamma h(z). \qedhere
\end{equation}
%
%
%
%
%
\end{proof}

Next we state a variation of Lemma 3.4 of Balasubramanian-Gun \cite{BG}, who worked instead with $g(z) = 1/h(z)$.
\begin{myprop}[\cite{BG}]
\label{prop:E2zeroshrationals}
If $E_2(z_0) = 0$ then $h(\gamma z_0) = \frac{a}{c}$ for $\gamma = (\begin{smallmatrix} a & b \\ c & d \end{smallmatrix})$.  Conversely, if $h(\tau_0) = \frac{a}{c}$ with coprime $a, c$, then $E_2(\gamma^{-1} \tau_0) = 0$ for $\gamma = (\begin{smallmatrix} a & b \\ c & d \end{smallmatrix})$.
\end{myprop}
 \begin{proof}
  Consider the case when $E_2(z_0)=0$ (so $h(z_0) = \infty$), and let $z=\gamma z_0$.  Note that $\gamma \infty= \frac{a}{c}$.
Then
  $$ h(\gamma z_0)=\gamma h(z_0)= \gamma \infty = \frac{a}{c}.$$
Conversely, suppose $h(\tau_0)=\frac{a}{c}$.  Then
\begin{equation}
 h(\gamma^{-1} \tau_0) = \gamma^{-1} h(\tau_0) = \gamma^{-1} \frac{a}{c} = \infty,
\end{equation}
so $E_2(\gamma^{-1} \tau_0) = 0$.
 \end{proof}

 \section{The real locus of $h$}
Let $\mh^* = \mh \cup \{ \infty \}$, $\mr^* = \mr \cup \{\infty \}$, and $\mq^* = \mq \cup \{ \infty \}$.  Then define
\begin{equation}
\mathcal{Q} = \{ z \in \mh^* : h(z) \in \mq^* \}, \quad \text{and} \quad \mathcal{R} = \{ z \in \mh^* : h(z) \in \mr^* \}.
\end{equation}
In other words, $\mathcal{Q} = h^{-1}(\mq^*)$, $\mathcal{R} = h^{-1}(\mr^*)$.  It follows easily from the definition of an equivariant function that $SL_2(\mz)$ acts on $\mathcal{Q}$ as well as $\mathcal{R}$.  Therefore, $\mathcal{Q}_D := \mathcal{Q} \cap D$ and $\mathcal{R}_D := \mathcal{R} \cap D$ have the properties that $SL_2(\mz) \mathcal{Q}_D = \mathcal{Q}$ and $SL_2(\mz) \mathcal{R}_D = \mathcal{R}$, so that to understand the shapes of $\mathcal{Q}$ and $\mathcal{R}$ it suffices to consider $\mathcal{Q}_D$ and $\mathcal{R}_D$.  Of course, $\mathcal{Q}_D \subset \mathcal{R}_D$, and as the latter set is simpler, we now focus on describing $\mathcal{R}_D$.  Numerically plotting the set $\mathcal{R}_D$ using Mathematica produces Figure \ref{fig:RDset}.  The left-hand image has the bottom arc $|z| = 1$ for reference, while the right-hand image is on a much finer scale.
\begin{figure}[h]
\includegraphics[width=0.4\textwidth]{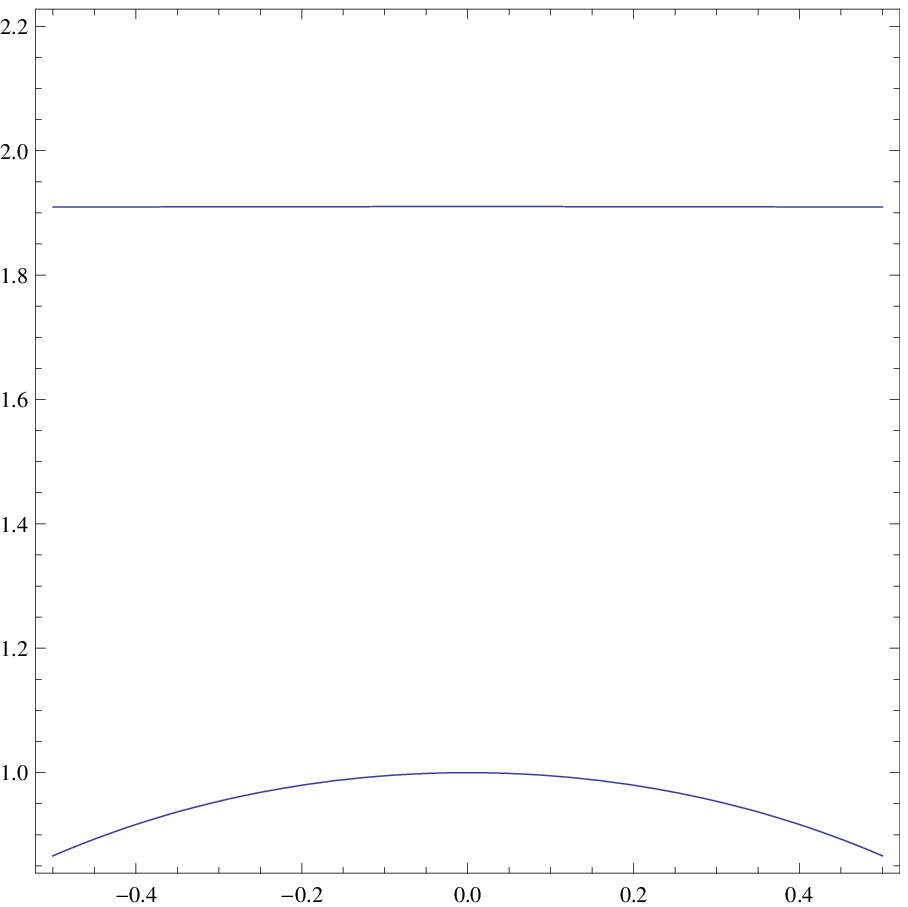}
\includegraphics[width=0.415\textwidth]{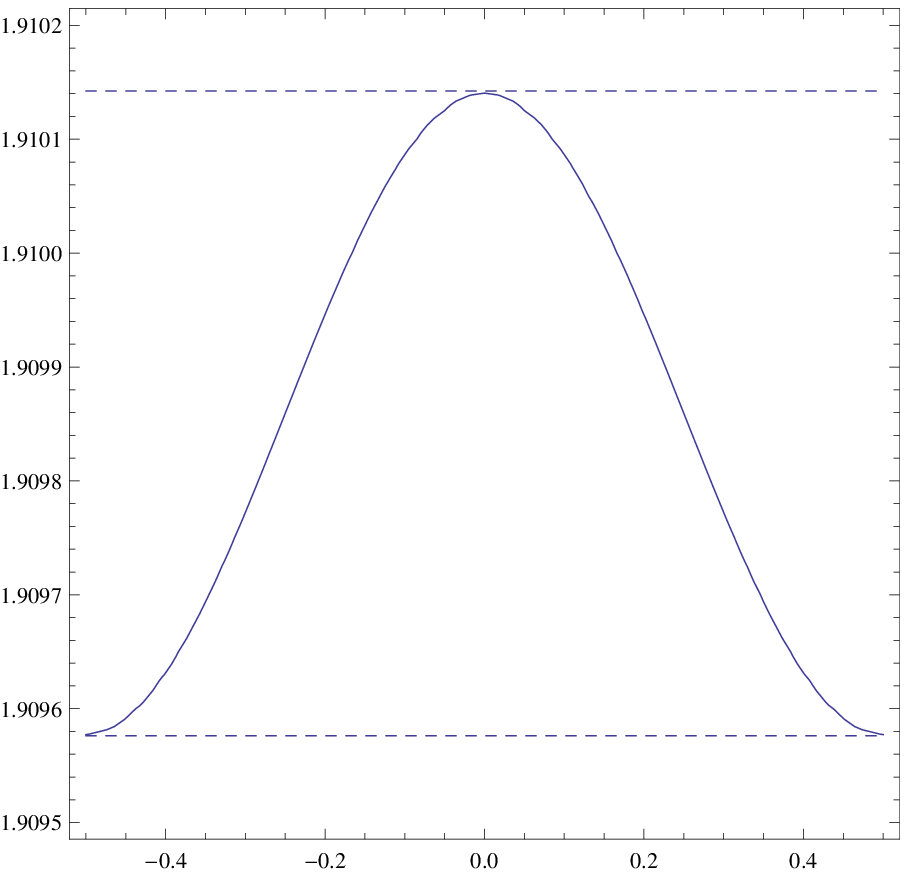}
\caption{The set $\mathcal{R}_D$}
\label{fig:RDset}
\end{figure}

\begin{mytheo}
\label{thm:RDsmallstrip}
The set $\mathcal{R}_D$ is contained in the small strip $|y- 6/\pi| < .000283$.
\end{mytheo}
That is, the set of $z \in D$ such that $h(z) \in \mr$ lies between two horocycles at heights $y =6/\pi \pm .000283$ ($y=1.90957\dots$ and $y=1.91014\dots$) which are shown as dotted lines in the right-hand image in Figure \ref{fig:RDset}.  Extending these two horocycles to $\mh$ by translating by all integers (that is, by the subgroup $\Gamma_{\infty}$) gives two circles around the point $\infty$.  Call the top circle $C_+$ and the bottom one $C_{-}$.  
The images of $C_{\pm}$ under $\gamma \in SL_2(\mz)$ are circles tangent to $\mr^*$ at $\gamma(\infty) \in \mq^*$.  
Thus $\gamma (\Gamma_{\infty} \mathcal{R}_D) = \gamma(\mz + \mathcal{R}_D)$ is squeezed between the two circles $\gamma C_{\pm}$ meeting at $\gamma(\infty) \in \mq^*$.  To the naked eye, the two circles $C_{+}$ and $C_{-}$ are almost indistinguishable, so that $\mathcal{R}$ appears to be a union of circles tangent to the rationals $-d/c$ with radii $\approx \frac12 \frac{\pi}{6 c^2}$.  For this, one calculates $\text{Im}(\gamma(x + \frac{6i}{\pi})) = \frac{6/\pi}{(6 c/\pi)^2 + (cx + d)^2}$, where $\gamma = (\begin{smallmatrix} a& b\\ c& d\end{smallmatrix})$, and this imaginary part is maximized when $cx + d= 0$, which gives the peak of the circle.

From Figure \ref{fig:RDset}, one can see that the constant $.000283$ is quite sharp.  This numerical constant arises because the first non-constant term in the Fourier expansion for $E_2$ gives $24 \frac{6}{\pi} e^{-2 \pi \frac{6}{\pi}} = .00028163\dots$.




\begin{mylemma}
\label{lemma:E2bound}
Suppose $y \geq \frac{\sqrt{3}}{2}$ and $N \geq 1$ is an integer.  Then
\begin{equation}
\label{eq:E2tailbound}
|E_2(z) - [1 - 24\sum_{n \leq N} \sigma_1(n) e(nz)]| \leq 24e^{-2\pi Ny} \Big(\frac{N^2}{2\pi y} + \frac{2N}{(2 \pi y)^2} + \frac{2}{(2 \pi y)^3}\Big).
\end{equation}
This bound is decreasing separately in $N \geq 1$ or $y \geq \frac{\sqrt{3}}{2}$.  Furthermore,
\begin{equation}
\label{eq:E2tailbound2}
|E_2(z) - 1| \leq 24 e^{-2 \pi y} \Big(1 + \frac{1}{2 \pi y} +  \frac{2}{(2 \pi y)^2} + \frac{2}{(2 \pi y)^3}\Big) < .14.
\end{equation}
\end{mylemma}
\begin{proof}
Let $r = e^{-2 \pi y}$, so in particular $r < .0044$.  Then the tail of the Fourier series for $E_2$ is bounded by $24 \sum_{n > N} \sigma_1(n) r^n$.  Obviously, $\sigma_1(n) \leq n^2$, and an exercise in calculus shows that $f_r(t) = t^2 r^t$ is monotone decreasing for $t \geq 1$ since $\log r < -\pi \sqrt{3} <-2$.  Therefore, the integral test gives
\begin{equation}
\label{eq:E2tailboundproof}
\sum_{n > N} n^2 r^n \leq  \int_N^{\infty} t^2 r^t dt = r^N \Big(\frac{N^2}{\log(1/r)} + \frac{2N}{\log^2(1/r)} + \frac{2}{\log^3(1/r)}\Big). 
\end{equation}
Re-writing this bound in terms of $r$ gives \eqref{eq:E2tailbound}.  The fact that the bound is decreasing in either $N$ or $y$ is evident from the integral formula in \eqref{eq:E2tailboundproof}.  We derive \eqref{eq:E2tailbound2} by choosing $N=1$ and directly bounding the $N=1$ term.
\end{proof}

\begin{proof}[Proof of Theorem \ref{thm:RDsmallstrip}]
Let $z \in D$, so $y \geq \frac{\sqrt{3}}{2}$.
Using \eqref{eq:E2tailbound2}, we have
\begin{equation}
h(z)= z + \frac{\frac{6}{\pi i}}{1 + (E_2(z) - 1)},
\end{equation}
which gives by a rearrangement
\begin{equation}
 h(z) - \big(z + \frac{6}{\pi i}\big) = \frac{6}{\pi i} \frac{1 - E_2(z)}{1 + (E_2(z) - 1)}.
\end{equation}
Therefore, for any $z \in D$ we have
\begin{equation}
\label{eq:hbound}
\Big|h(z) - \big(z + \frac{6}{\pi i}\big)\Big| \leq \frac{6}{\pi} \frac{|E_2(z) - 1|}{1 - |E_2(z) - 1|}.
\end{equation}
If $h(z) \in \mr$, we then obtain
\begin{equation}
\label{eq:ybound}
 \Big|y - \frac{6}{\pi}\Big| \leq \frac{6}{\pi} \frac{|E_2(z) - 1|}{1 - |E_2(z) - 1|}.
\end{equation}
To start a recursive process, we use the crude bound $|E_2(z) - 1| < .14$ (from \eqref{eq:E2tailbound2}) which holds throughout the fundamental domain, obtaining that
\begin{equation}
|y - \frac{6}{\pi}| < \frac{6}{\pi} \frac{.14}{.86} < .32.
\end{equation}
This is weaker than the statement of Theorem \ref{thm:RDsmallstrip}, but we shall iterate this to improve on the bounds.  Using
$y \geq \frac{6}{\pi} - .32$ in \eqref{eq:E2tailbound2}, we obtain $|E_2(z) - 1| <.0013\dots$ which we can re-insert into \eqref{eq:ybound}, giving the improved estimate $|y-\frac{6}{\pi}| < .0024$.  Using this iteration once more leads to $|E_2(z) - 1| < .00017$, and $|y-\frac{6}{\pi}| < .00032$.  At this point we take $N = 4$ and use
\begin{equation}
 |E_2(z) - 1| \leq 24 \Big(\sum_{n \leq 4} \sigma_1(n) e^{-2\pi ny} + e^{-8\pi y}\Big(\frac{4^2}{2\pi y} + \frac{8}{(2 \pi y)^2} + \frac{2}{(2 \pi y)^3} \Big) \Big),
\end{equation}
which for $y > \frac{6}{\pi} - .00032$ gives $|E_2(z) - 1| < .00015$, and finally $|y-\frac{6}{\pi}| < .000283$.  
\end{proof}

\section{The rational values of $h$}


In light of Proposition \ref{prop:E2zeroshrationals}, we now turn to the rational values of $h$ as a means to understand the location of the zeros of $E_2$.  Recall that each zero of $E_2$ is $SL_2(\mz)$-equivalent to a rational value of $h$.  Figure \ref{fig:RDsetWithzeros} shows the points in the fundamental domain equivalent to the zeros of $E_2$ with $y > .002$
\begin{figure}[h]
\includegraphics[width=0.65\textwidth]{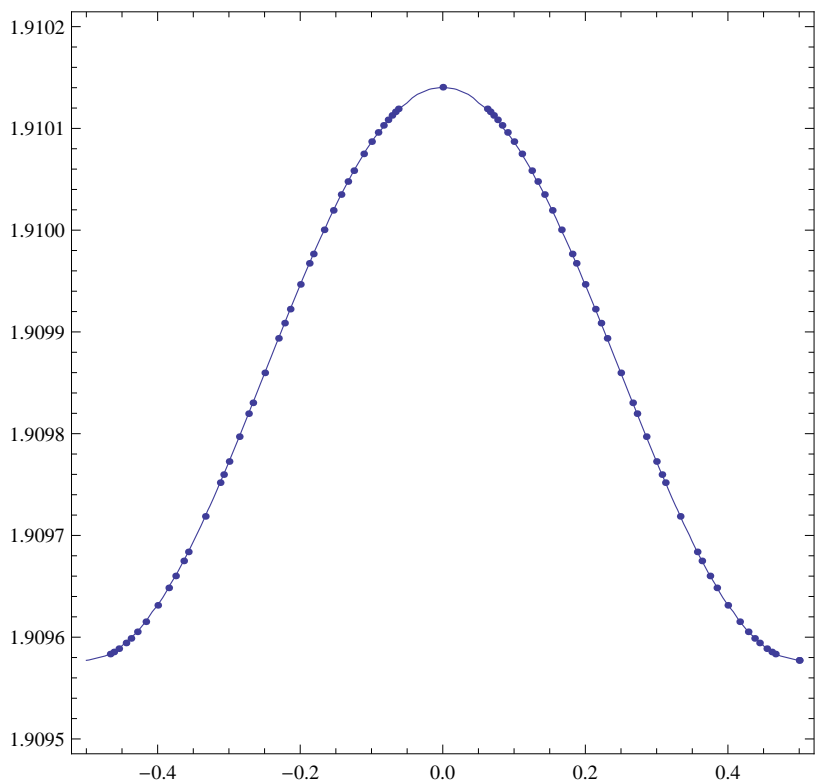}
\caption{Rational values of $h$ with small denominators overlaid on $\mathcal{R}_D$}
\label{fig:RDsetWithzeros}
\end{figure}

\begin{myprop}
 The function $h$ restricted to the region $\{ z \in \mh : y > .95\}$ is injective and takes every real value in the interval $[-1/2, 1/2]$ exactly once.  In particular, $h$ takes every rational value in $[-1/2, 1/2]$ exactly once.
\end{myprop}
It would be interesting to determine if $h$ is injective on the entire fundamental domain $D$.  For our purposes, we were interested in the real values which occur for $y \approx 6/\pi$ so the behavior at the bottom of the fundamental domain was not relevant. 
\begin{proof}
See Figure \ref{fig:himage} for a picture of the image of the fundamental domain $D$ under $h$. 
\begin{figure}[h]
\includegraphics[width=0.25\textwidth]{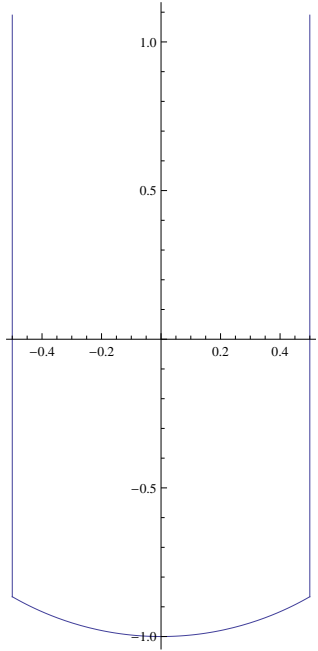}
\caption{The image of the fundamental domain $D$ under $h$}
\label{fig:himage}
\end{figure}

We now explain rigorously some relevant features of the picture.  Firstly, $h(\pm 1/2 + iy)$ has real part $\pm 1/2$, which follows easily from the fact that $E_2(\pm 1/2 + iy) \in \mr$.  Furthermore, we claim that with $\rho = e^{2 \pi i/3}$, $h(\rho) = \overline{\rho}$, $h(i) = -i$, and $h(1-\overline{\rho}) = 1- \rho$.  In general, if $\tau$ is an elliptic point, and $\gamma \in SL_2(\mz)$ fixes $\tau$, that is, $\gamma \tau = \tau$, $\gamma \neq 1$, then $h(\gamma \tau) = h(\tau) = \gamma h(\tau)$, so we immediately deduce that either $h(\tau) = \tau$ or $h(\tau) = \overline{\tau}$.  However, by a direct numerical calculation one can rule out the case that $h(\tau) = \tau$ for the three elliptic points, since $h(\tau) \approx \tau - \frac{6i}{\pi}$ (using $E_2(\tau) \approx 1$).  See also \cite[(9)-(10)]{BS} or Proposition 5.6 of \cite{SaSe} for this result on elliptic fixed points.
Also we note the simple fact that $\lim_{y \rightarrow \infty} h(\pm 1/2 + iy) = \infty$.  Taken together, this discussion means that $h$ maps the left and right sides of $D$ to vertical line segments from $\pm 1/2 - i \frac{\sqrt{3}}{2}$ to $\pm 1/2 + i \infty$.  
By a numerical evaluation of $h$ along the bottom arc $z = e^{i \theta}$, $\frac{\pi}{3} \leq \theta \leq \frac{2 \pi}{3}$, we obtain Figure \ref{fig:himage}.  We conclude that $h(\mathcal{R}_D) = [-1/2, 1/2]$, that is, on $D$, $h$ takes on every real value in $[-1/2, 1/2]$ at least once, and no other real values.

Finally, we argue that $h$ is injective for $y \geq .95$, which then implies $h$ takes every real (and hence rational) value exactly once.  We first claim that for $y \geq .95$, we have
\begin{equation}
\label{eq:h'estimate}
 |h'(z)  -1 | < .89.
\end{equation}
An exact formula for $h'$ is given by
\begin{equation}
 h'(z)  = 1 + \frac{6i}{\pi} \frac{E_2'(z)}{E_2(z)^2}.
\end{equation}
We have $E_2'(z) = 48\pi i \sum_{n=1}^{\infty} \sigma_1(n) n e(nz)$, whence by a computer calculation (safely taking say $100$ terms in the Fourier expansion)
\begin{equation}
|E_2'(x+iy)| \leq 48 \pi \sum_{n=1}^{\infty} \sigma_1(n) n e^{-2 \pi ny} < .4,
\end{equation}
and furthermore $|E_2(z) - 1| \leq .07$ (again, by a direct computer calculation), so $|E_2(z)| \geq .93$.
Thus
\begin{equation}
 |h'(z) - 1| < \frac{6}{\pi} \frac{.4}{(.93)^2} < .89.
\end{equation}

Next we argue that \eqref{eq:h'estimate} implies that $h$ is injective.
Suppose that $f$ is a holomorphic function on a convex set $C$ and that $|f'(z) - 1| \leq \delta < 1$ for all $z \in C$.  We prove that this implies $f$ is injective on $C$.
To see this, we begin by noting
\begin{equation}
 f(z_2) - f(z_1) = (z_2 - z_1) + \int_{z_1}^{z_2} (f'(s) - 1) ds,
\end{equation}
where the contour is a straight line segment connecting $z_1$ and $z_2$. Therefore, if $f(z_2) = f(z_1)$, we have
\begin{equation}
 |z_2 - z_1| \leq \int_{z_1}^{z_2} |f'(s) - 1| |ds| \leq \delta  |z_2 - z_1|,
\end{equation}
which is a contradiction unless $z_2 = z_1$.  That is, $f$ is injective on $C$.
This immediately implies that $h$ is injective for $y \geq .95$.
\end{proof}

Now we are ready to give the more precise version of Theorem \ref{thm:zeroapprox}.
\begin{mytheo}
\label{thm:locationofzeros}
Let $v_0 = 6/\pi$. Suppose that for $\tau_0 \in D$, $h(\tau_0) = \frac{a}{c}$ with $(a,c) = 1$, and let $\gamma = (\begin{smallmatrix} a & b\\ c & d \end{smallmatrix}) \in SL_2(\mz)$ with $|d/c| \leq 1/2$.  Then the zero $z_0 = \gamma^{-1} \tau_0$ of $E_2$ associated to $\tau_0$ (as in Proposition \ref{prop:E2zeroshrationals}) satisfies
 \begin{equation}
 \label{eq:locationofzeros}
  z_0 = - \frac{d}{c} + \frac{i}{c^2 v_0} + \frac{\theta}{c^2 v_0^2},
 \end{equation}
 where $\theta \in \mc$ satisfies $|\theta| < .000283$.
\end{mytheo}
Examples.  Let $\widehat{z}_c^d = -\frac{d}{c} + \frac{i}{c^2 v_0}$ be the approximate zero of $E_2$. 
\begin{equation}
\begin{aligned}
 h(\tau_0) = 0, \quad \gamma = \big(\begin{smallmatrix} 0 & -1 \\ 1 & 0 \end{smallmatrix}\big), \quad &\widehat{z}_1^0 = 0 + .523599 i, 
 \\
 h(\tau_0) = 1/2, \quad \gamma = \big(\begin{smallmatrix} 1 & 0 \\ 2 & 1 \end{smallmatrix}\big), \quad &\widehat{z}_2^1 = -\tfrac12 + .130899 i, 
 \\
 h(\tau_0) = 1/3, \quad \gamma = \big(\begin{smallmatrix} 1 & 0 \\ 3 & 1 \end{smallmatrix}\big), \quad &\widehat{z}_3^1 = -\tfrac13 + .0581776 i, 
 \\
 h(\tau_0) = 1/4, \quad \gamma = \big(\begin{smallmatrix} 1 & 0 \\ 4 & 1 \end{smallmatrix}\big), \quad &\widehat{z}_4^1 = -\tfrac14 + .0327249 i \\
 h(\tau_0) = 1/5, \quad \gamma = \big(\begin{smallmatrix} 1 & 0 \\ 5 & 1 \end{smallmatrix}\big), \quad &\widehat{z}_5^1 = -\tfrac15 + .020944 i \\
 h(\tau_0) = -2/5, \quad \gamma = \big(\begin{smallmatrix} -2 & 1 \\ 5 & 2 \end{smallmatrix}\big), \quad &\widehat{z}_5^2 = -\tfrac25 + .020944 i.
\end{aligned}
\end{equation}

The zero $\widehat{z}_c^d$ above occurs in a geometrically special location as we now describe.  Consider the circle around infinity $C= \{ x + \frac{6i}{\pi} : x \in \mr \}$ (which is very close to the circles $C_{\pm}$ discussed following Proposition \ref{prop:E2zeroshrationals}).  The circle $\gamma^{-1} C$ with  $\gamma = (\begin{smallmatrix} a & b\\ c & d \end{smallmatrix}) \in SL_2(\mz)$ is tangent to the real line at $x=-\frac{d}{c}$ and has height ($=$ diameter) $\frac{1}{c^2 v_0}$.  The point mapped to the top of the circle is $\widehat{\tau}:=\frac{a}{c} + \frac{6i}{\pi}$, that is, $\gamma^{-1} (\frac{a}{c} + \frac{6i}{\pi}) = -\frac{d}{c} + \frac{i}{c^2 v_0}$, which is $\widehat{z}_c^d$, the first-order approximation on the right hand side of \eqref{eq:locationofzeros}.  In effect, the zero $z_0$ in \eqref{eq:locationofzeros} is very close to the peaks of the circles $\gamma^{-1} C_{\pm}$.  See Figure \ref{fig:zerosoncircles} for a picture.
\begin{figure}[h]
\includegraphics[width=0.5\textwidth]{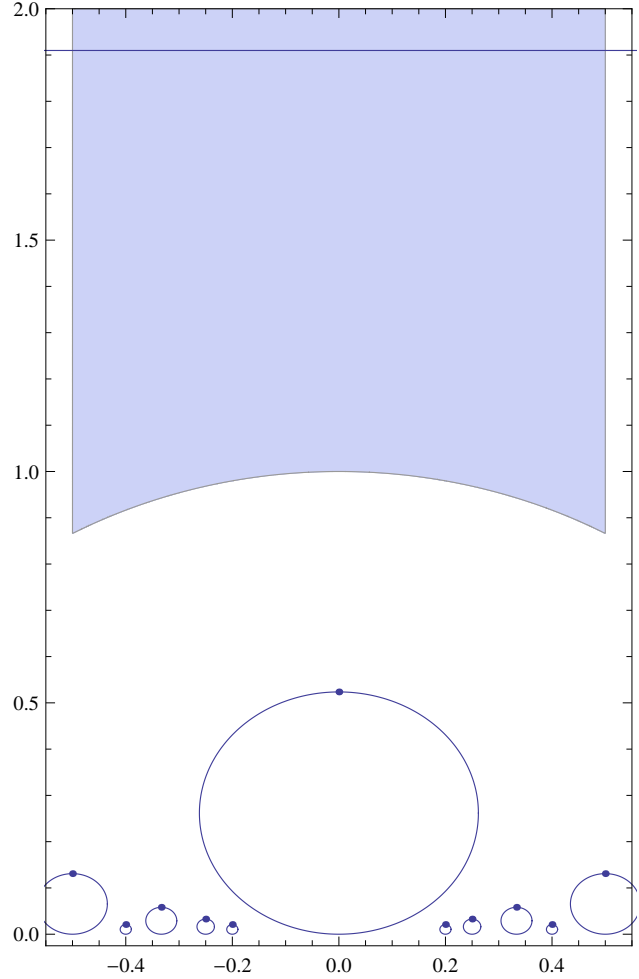}
\caption{The zeros of $E_2$ overlaid on $\mathcal{R}$}
\label{fig:zerosoncircles}
\end{figure}

In fact, we can prove a more precise approximation.
\begin{mytheo}
\label{thm:locationofzeros2}
 Let conditions and notation be as in Theorem \ref{thm:locationofzeros}.  Let $\lambda_0 = 24 \frac{6}{\pi} e^{-2 \pi v_0} = .000281\dots$.  Then the zero $z_0$ satisfies
 \begin{equation}
  z_0 = -\frac{d}{c} + \frac{\lambda_0 \sin(2 \pi \frac{a}{c})}{c^2 v_0^2} + i \frac{1 - \frac{\lambda_0}{v_0} \cos(2 \pi \frac{a}{c})}{c^2 v_0} + O(c^{-2} e^{-4 \pi v_0}),
 \end{equation}
 where the implied constant is absolute.
\end{mytheo}
For simplicity, we chose not to explicitly bound the error term in this expression as we did in Theorem \ref{thm:locationofzeros}, but this could be done with more effort.  In practice, Theorem \ref{thm:locationofzeros2} has a few more digits of accuracy than that of Theorem \ref{thm:locationofzeros}.
The proof is a recursive linearization argument similar to Newton's method, which could be extended to give more digits of accuracy at the expense of having increasingly complicated expressions.

\begin{proof}[Proofs]
We begin with Theorem \ref{thm:locationofzeros}.
If $h(\tau_0) = \frac{a}{c}$, then from \eqref{eq:hbound} and the fact that $|E_2(\tau_0) - 1| < .00014\dots$ (which arose in the proof of Theorem \ref{thm:RDsmallstrip}), we have
\begin{equation}
 \Big|\tau_0 - \frac{a}{c} - \frac{6i}{\pi}\Big| \leq \frac{6}{\pi} \frac{|1- E_2(\tau_0)|}{1 - |E_2(\tau_0)-1|} < \varepsilon, 
\end{equation}
with $\varepsilon = .0002821\dots$.  That is, with $\widehat{\tau_0} = \frac{a}{c} + \frac{6i}{\pi}$, we have $|\tau_0 - \widehat{\tau_0}| < \varepsilon$.

By a direct calculation, we have with $\alpha = \gamma^{-1} =  (\begin{smallmatrix} d & -b\\ -c & a \end{smallmatrix})$
\begin{equation}
\alpha \widehat{\tau_0} = \frac{d(\frac{a}{c} + iv_0) - b}{-c(\frac{a}{c} + iv_0) + a} = \frac{\frac{1}{c} + div_0}{-civ_0} = -\frac{d}{c} + \frac{i}{c^2 v_0} =: \widehat{z_0}.
\end{equation}
For any $\alpha  \in SL_2(\mz)$, the following elementary identity is valid:
 \begin{equation}
  \alpha u - \alpha w = \frac{u-w}{j(\alpha, u) j(\alpha, w)}.
 \end{equation}
We apply this with $w = \widehat{\tau_0}$ 
and $u = \tau_0 = \frac{a}{c} + \frac{6i}{\pi} + \delta$ (so $\delta \in \mc$ and $|\delta| < \varepsilon = .0002821\dots$), and $\alpha  = (\begin{smallmatrix} d & -b\\ -c & a \end{smallmatrix}) = \gamma^{-1}$.  Now $h(\alpha \tau_0) = \alpha \frac{a}{c} = \infty$, so $\alpha \tau_0 = z_0$ is the zero of $E_2$ corresponding to this rational value of $h$, as in Proposition \ref{prop:E2zeroshrationals}.  We thus have
\begin{equation}
\label{eq:alphaz0approximation}
z_0= \alpha \tau_0 = \alpha \widehat{\tau_0} + \frac{\delta}{(-civ_0)(-civ_0 + c \delta)} = -\frac{d}{c} + \frac{i}{c^2 v_0} - \frac{\delta}{c^2 v_0^2(1 - \frac{\delta}{i v_0})}.
\end{equation}
Thus
\begin{equation}
\Big|z_0 - \big(-\frac{d}{c} + \frac{i}{c^2 v_0}\big)\Big| < \frac{|\delta|}{c^2 v_0^2} \frac{1}{1-|\frac{\delta}{v_0}|} < \frac{.000283}{c^2 v_0^2}.
\end{equation} 
This gives Theorem \ref{thm:locationofzeros}.
 
 Now we prove Theorem \ref{thm:locationofzeros2}.  The main difference here is simply that we include the next term in the Fourier expansion of $E_2(z)$.  For $z \in \mathcal{R}_D$, we have
 \begin{equation}
 |E_2(z) - (1 - 24 e(x) e^{-2 \pi y})| < 2.73 \times 10^{-9},
 \end{equation} 
 which is approximately $24 \sigma_1(2) e^{-4 \pi v_0} \approx 2.72 \times 10^{-9} $.
 Then we have
 \begin{equation}
 h(z) = z + \frac{\frac{6}{\pi i}}{1 - 24 e(x) e^{- 2\pi y} + \eta} = z + \frac{6}{\pi i} \Big(1 + 24e(x) e^{-2\pi y}- \eta + \frac{(24e(x) e^{-2\pi y}- \eta)^2}{1 - 24 e(x) e^{- 2\pi y} + \eta} \Big),
 \end{equation}
 where $\eta \in \mc$ satisfies $|\eta| < 2.73 \times 10^{-9}$.
 Taking a Taylor expansion, we have
 \begin{equation}
 h(z) = z + \frac{6}{\pi i} + 24 \frac{6}{\pi i} e^{-2\pi v_0} e(x) e^{-2 \pi (y-v_0)} + O(e^{-4 \pi v_0}).
 \end{equation}
 Using $|y-\frac{6}{\pi}| = |y-v_0| < .000283 = O(e^{-2 \pi v_0})$, we obtain that for $z \in \mathcal{R}$, 
 \begin{equation}
 h(z) = z + \frac{6}{\pi i} + \frac{\lambda_0}{i} e(x) +O(e^{-4 \pi v_0}).
 \end{equation}
 Grouping by the real and imaginary parts, it becomes
 \begin{equation}
 h(z) = x + \lambda_0 \sin(2 \pi x) + i(y -\frac{6}{\pi} - \lambda_0 \cos(2 \pi x)) + O(e^{-4 \pi v_0}).
 \end{equation}
 Setting $h(z) = \frac{a}{c}$, we obtain
 \begin{align*}
 x + \lambda_0 \sin(2 \pi x) = \frac{a}{c} + O(e^{-4 \pi v_0}) \\
 y - \frac{6}{\pi} - \lambda_0 \cos(2 \pi x) = O(e^{-4 \pi v_0}) .
 \end{align*}
 To first approximation, $x = \frac{a}{c} + O(e^{-2 \pi v_0})$ and $y = \frac{6}{\pi} + O(e^{-2 \pi v_0})$, and since $\lambda_0 = O(e^{-2 \pi v_0})$, we may use this first-order approximation inside $\cos$ and $\sin$.  That is, we have the approximations
 \begin{align*}
 x = \frac{a}{c} - \lambda_0 \sin(2 \pi \frac{a}{c}) + O(e^{-4 \pi v_0}) \\
 y = \frac{6}{\pi} + \lambda_0 \cos(2 \pi \frac{a}{c}) + O(e^{-4 \pi v_0}).
 \end{align*}
 That is, the $\delta$ appearing in \eqref{eq:alphaz0approximation} is given by $\delta = - \lambda_0 \sin(2 \pi \frac{a}{c}) + i \lambda_0 \cos(2 \pi \frac{a}{c}) + O(e^{-4 \pi v_0})$, so we now have
 \begin{equation}
 z_0 = -\frac{d}{c} + \frac{\lambda_0 \sin(2 \pi \frac{a}{c})}{c^2 v_0^2} + i \frac{1- \frac{\lambda_0}{v_0} \cos(2 \pi \frac{a}{c})}{c^2 v_0^2} + O(e^{-4 \pi v_0}). \qedhere
 \end{equation}
\end{proof}



\begin{thebibliography}{99}
\bibitem[BG]{BG} R.
Balasubramanian and S. Gun,
\emph{On zeros of quasi-modular forms.}
J. Number Theory 132 (2012), no. 10, 2228--2241.

\bibitem[DJ]{DukeJenkins}  W. Duke and P. Jenkins, \emph{On the zeros and coefficients of certain weakly holomorphic modular forms.} Pure Appl. Math. Q. 4 (2008), no. 4, Special Issue: In honor of Jean-Pierre Serre. Part 1, 1327--1340.

\bibitem[ElBS]{BS} A. El Basraoui, and A. Sebbar,
\emph{Zeros of the Eisenstein series $E_2$}.
Proc. Amer. Math. Soc. 138 (2010), no. 7, 2289--2299.

\bibitem[GS]{GhoshSarnak} A. Ghosh and P. Sarnak, 
\emph{Real zeros of holomorphic Hecke cusp forms.}
J. Eur. Math. Soc. 14 (2012), no. 2, 465--487. 
\bibitem[H]{Hahn}  H. Hahn, \emph{On zeros of Eisenstein series for genus zero Fuchsian groups.} Proc. Amer. Math. Soc. 135 (2007), no. 8, 2391--2401.
\bibitem[HS]{HolowinskySound} R. Holowinsky and K. Soundararajan, 
\emph{Mass equidistribution for Hecke eigenforms.}
Ann. of Math. (2) 172 (2010), no. 2, 1517--1528. 
\bibitem[IJT]{IJT} {\"O}. Imamo{\=g}lu, J. Jermann, and {\'A}. T{\'o}th, \emph{Estimates on the zeros of $E_2$}, arxiv.org/abs/1312.3119.
\bibitem[MNS]{MNS} T. Miezaki, H. Nozaki, and J. Shigezumi, \emph{On the zeros of {E}isenstein series for {$\Gamma^*_0(2)$} and {$\Gamma^*_0(3)$}} J. Math. Soc. Japan 59 (2007), no. 3, 693--706.
\bibitem[RS]{RS} F. K. C. Rankin,  H. P. F Swinnerton-Dyer,
\emph{On the zeros of Eisenstein series.}
Bull. London Math. Soc. 2 1970, 169--170.
\bibitem[R]{Rudnick} Z. Rudnick, \emph{
On the asymptotic distribution of zeros of modular forms.}
Int. Math. Res. Not. 2005, no. 34, 2059--2074.
\bibitem[SaSe]{SaSe} H. Saber and Abdellah Sebbar, \emph{On the critical points of modular forms},  J. Number Theory 132 (2012), no. 8, 1780--1787.
\bibitem[SeSe]{SeSe} Abdellah Sebbar and Ahmed Sebbar,
\emph{Equivariant functions and integrals of elliptic functions.}
Geom. Dedicata 160 (2012), 373--414.
\bibitem[Z]{Zagier} D. Zagier, \emph{
Elliptic modular forms and their applications.} The 1-2-3 of modular forms, 1--103,
Universitext, Springer, Berlin, 2008. 
\end{thebibliography}
\end{document}